\documentclass[11pt,twoside]{article}
\usepackage{a4}

\author{Gangavarapu, Venkata Krishna Kishore}

\usepackage{amsfonts,amscd,amssymb,amsmath,amsthm,latexsym}

\usepackage{tikz-cd}
\usepackage{hyperref}
\usepackage{enumerate}
\usepackage{braket}
\usepackage{hyperref}
\usepackage[all]{xy}

\def\F{\mathbb{F}}

\def\Z{\mathbb{Z}}

\newcommand{\CalA}{\mathcal{A}}

\newcommand{\CalO}{\mathcal{O}}

\theoremstyle{plain}
\newtheorem{thm}{Theorem}[section]

\newtheorem{cor}[thm]{Corollary}

\newtheorem{prop}[thm]{Proposition}
\newtheorem{lem}[thm]{Lemma}

\newtheorem{defn}[thm]{Definition}

\newtheorem{remark}[thm]{Remark}

\newtheorem*{claim*}{Claim}
\numberwithin{equation}{section}

\newcommand{\Gam}{\Gamma}
\newcommand{\gam}{\gamma}

\newcommand{\Hom}{\operatorname{Hom}}

\newcommand{\aut}{\operatorname{Aut}}
\newcommand{\Ker}{\operatorname{Ker}}

\newcommand{\ind}{\operatorname{ind}}

\newcommand{\abs}[1]{| #1 |}

\newcommand{\zp}{\Z_p}

\newcommand{\notty}{\mathfrak{N}}

\newcommand{\powerseries}{\F_p[\![t ]\!]}
\newcommand{\laurentseries}{\F_p(\!(t )\!)}

\newcommand{\maxi}{\maxideal}

\newcommand{\nmaxi}{\mathfrak{M}}

\newcommand{\npg}{U_1}
\newcommand{\maxideal}{\mathfrak{M}}

\newcommand{\zmodpnz}{\Z/p^n\Z}

\newcommand{\npgpower}[1]{U_{#1}}
\newcommand{\FZ}{\mathfrak{Z}}
\newcommand{\FM}{\mathfrak{M}}
\newcommand{\genbs}{\langle b^{(0)}, \ldots, b^{(n-1)} \rangle}
\newcommand{\pmodp}{\pmod{p}}
\newcommand{\pmodpp}{\pmod{p^2}}
\newcommand{\al}{\alpha}
\newcommand{\ub}[1]{{}_{#1}}

\begin{document}

\setcounter{page}{1}                                 

\markboth {\hspace*{-9mm} \centerline{\footnotesize \sc
   $p^2$-torsion in Nottingham group}
                 }
                { \centerline                           {\footnotesize \sc  
                   Krishna Kishore                                               } \hspace*{-9mm}              
               }

\begin{center}
{ 
       {\Large \textbf { \sc 
    Torsion elements of the Nottingham group of order $p^2$ and type $\langle 2, m \rangle$
                               }
       }
\\

\medskip

{\sc Krishna Kishore \footnote{The author is supported by DST-SERB (Science and Engineering Research Board) national postdoctoral fellowship of the Government of India [SERB-PDF/2016/004056].}}\\
{\footnotesize 
Department of Mathematics,
Indian Institute of Science Education and Research,
Pune, India.\\}
{\footnotesize e-mail: {\it venkatakrishna@iiserpune.ac.in}}
}
\end{center}

\thispagestyle{empty}

\hrulefill

\begin{abstract}  
{\footnotesize \noindent We classify torsion elements of order $p^2$ and type $\langle 2, m \rangle$ in the Nottingham group defined over a prime field of characteristic $p >0$.
}
 \end{abstract}
 \hrulefill



\section{Introduction}\label{intro}

Let $p$ be a prime number, and $\F_p$ the prime field of characteristic $p$. Let  $K := \laurentseries$ be the field of Laurent-series in one variable $t$ over $\F_p$,  and $\CalO_K := \powerseries$ its ring of integers. Let $U_1 := 1 + t \CalO_K $ be the group of principal units, and $U_k = 1 + t^k \CalO_K$, $k \geq 2$, be its higher unit  subgroups. Consider the group $\aut(K)$ of automorphisms of $K$, and the subgroup $\aut^1(K)$ consisting of \textit{wild} automorphisms of $K$, namely those which map the uniformizer $t$ to the product $tz$ for some $z \in U_1$. The set $\Set{t z | z \in U_1}$, corresponding to $\aut^1(K)$, equipped with composition of power series is a group, called the \textit{Nottingham group} $\notty$ over $\F_p$; see \cite{Ca}, \cite{Ca1}, \cite{Je}, \cite{DF}.

It is well-known that the Nottingham group $\notty$, equipped with profinite topology, contains every countably based pro-$p$-groups as a closed subgroup  \cite{Ca}, whence interest in its study, as it offers a good model to test conjectures on pro-$p$-groups. The classification of elements of order $p$ in the Nottingham group defined over any finite field $\kappa$ is due to Klopsch \cite{Kl}, according to which they are classified in two steps. First, according to the depth, which may be any positive integer $d$ prime to $p$; second, when $d$ is given, there are only $\abs{\kappa}-1$ conjugacy classes, described by the coefficient of the series in degree $d+1$. On the other hand, Lubin gave a different proof for the theorem of Klopsch using a different framework, namely that of associating characters $U_1 \to \Z/p^n \Z$ to conjugacy class of torsion elements in $\notty$ and order $p^n$. The setup allows one to compute torsion elements of order $p^n$ for any integer $n \geq 1$ \cite{Lu}. In this article we use this setup for the analysis of break sequences of strict equivalence classes. To state our main results, we now introduce some terminology. 

Consider the $\Z_p$-module $\Hom_{\Z_p}^{\textrm{cont}}(U_1, \Z/p^n \Z)$ of continuous characters $\chi: U_1 \to \Z/p^n \Z$ on which the group $\notty$ acts on the right in a manner compatible with $\Z_p$-module structure: 
$$
{}_w \chi(f(t)) := \chi( f \circ w(t)) 
$$
for $f(t) \in \npg = 1 + t \powerseries$ and $w \in \notty$.  Two characters $\chi, \psi$ are called strictly equivalent, denoted $\chi \simeq \psi$, if there exists a $u \in \notty$ such that $\psi = \ub{u} \chi$ and $u(t)/t \in \ker \chi$; if the former condition holds they are said to be weakly equivalent, and denoted $\chi \sim \psi$.  To any surjective character $\chi : U_1 \to \Z/p^n \Z$ there corresponds a sequence of integers $\langle b^{(0)}, \ldots b^{(n-1)} \rangle$, called \textit{the break sequence of $\chi$ or type of $\chi$}, where $b^{(j-1)}$ is the largest positive integer $b$ such that 
$$
\chi(1 + \maxi^{b+1}) \subset p^j \Z/p^n \Z \; \textrm{  but  } \; \chi(1 + \maxi^b) \not \subset p^j \Z/p^n \Z.
$$
Due to the work of Lubin \cite{Lu}, the conjugacy classes of torsion elements of order $p^2$ in $\notty$ correspond to strict equivalence classes of surjective continuous characters. Since two strictly equivalent characters have the same break sequence, one may pass onto the analysis of break sequences. In an unpublished work, Lubin classified torsion elements of order $p^2$ of break sequence type $\langle 1, m \rangle$ over \textit{any} finite field; see Theorem \ref{Lubin_<1,m>}.  The main goal of this article is to classify elements of order $p^2$ of type $\langle 2,m \rangle$ over any \textit{prime} finite field.

Let $E_i$ be a topological basis of the principal unit group $U_1 = 1 + \FM$, and $\FZ_i$ be the basis of $\Hom_{\Z_p}^{\textrm{cont}}(1 + \FM, \Z/p^2 \Z)$ dual to $E_i$. Let $\chi$ be a character of type $\langle 2, m \rangle$ with the following expansion:
\begin{equation}\label{intro_eq}
\chi = x_1 \FZ_1 + x_2 \FZ_2 + \sum_{\substack{1 \leq j \leq m \\ p \nmid j}} a_j. p \FZ_j,
\end{equation}
where all the coefficients are in $\Z/p^2 \Z$ which have representatives between $0$ and $p-1$ and $x_2, a_m$ are nonzero; see Lemma \ref{bs_lemma}. As it turns out, only some of these coefficients together with a relation among them determine the character up to strict equivalence.  Accordingly we collect these coefficients into \textit{indicator} of the character $\chi$, denoted $\ind \chi$, and defined based on $m \pmodp$, namely that $\ind \chi =  [x_2,\frac{x_1}{x_2} -  \frac{ a_{m-1}^p}{(m-1)x_2^p}], [ x_2, a_m, \frac{x_1}{x_2} -  \frac{ a_{m-1}}{(m-1)a_m}]$, resp. $[x_2, a_m]$ corresponding to $m \equiv 0 \pmodp$, $m \not \equiv 0,1 \pmodp$, and $m \equiv 1 \pmodp$ respectively. Now we can state some of our main results.

\begin{thm}\label{intro_my_main_thm}
Let $\chi, \psi$ be characters of order type $\langle 2, m \rangle$.  Then $\chi$ is weakly equivalent to $\psi$ if and only if  $\ind \psi = \ind \chi$.
\end{thm}

The proof of the theorem is briefly as follows. We choose the basis $\Set{E_j := 1 + t^j}$, $j \geq 1$ and $p \nmid j$, of the principal unit group $U_1$ and correspondingly the canonical basis $\FZ_i$ dual to $E_j$. Necessity of the conditions follow immediately upon evaluating the character at various basis elements $E_j$. In order to prove the sufficiency condition, we use backwards induction and progressively build characters starting from $\chi$ and at each stage we find appropriate $u \in \notty$ such 
all the coefficients of $\ub{u} \chi$ and $\psi$ with index $k$ for all $j \leq k \leq m$ are equal. The difficulty is in finding appropriate $u$ that satisfies  this criterion, so by a careful analysis of the relations among the coefficients we obtain the desired $u \in \notty$.  As a consequence of Theorem \ref{intro_my_main_thm}, we obtain the main result of this article:

\begin{thm}\label{intro_final_thm}
Let $d_m$ denote the number of conjugacy classes of Nottingham elements of order $p^2$ and type $\langle 2, m \rangle$. Then
\[
d_m^{\textrm{weak}} \leq d_m \leq p d_m^{\textrm{weak}}
\]
where
\[
d_m^{\textrm{weak}} = 
\begin{cases}
p(p-1)	 & \; \textrm{if} \; m \equiv 0 \pmodp. \\
(p-1)^2	 & \; \textrm{if}   \; m \equiv  1 \pmodp.\\
p(p-1)^2	 & \; \textrm{otherwise}. \\
\end{cases}
\]  
\end{thm}

Finally, we refer the reader to an elegantly written article  about Nottingham group by Sautoy and Fesenko \cite{DF}. Also the reader may refer to an article of Chinburg and Symonds, where in, the authors have computed an element of order $4$ in the Nottingham group at the prime $2$ \cite{Ch}.

The outline of the article is as follows. In \S \ref{not} we explain the notation and conventions used in this article. In \S \ref{Lubin_article} we describe Lubin's framework which will be used later to classify torsion elements of order $p^2$; for more details of the approach, we recommend Lubin's work in \cite{Lu}. In \S \ref{top_str} we discuss the topological structure of the principal unit group $U_1$ and a lemma of Lubin. In \S \ref{str_char} we prove several computational lemmas that describe relations among the coefficients of the characters  that are essential to establish  Theorem \ref{intro_my_main_thm} and Theorem \ref{intro_final_thm} (same as Theorem \ref{my_main_thm} of \S \ref{main_section}, Theorem  \ref{final_thm} of \S \ref{classification_strict_classes} correspondingly). In \S \ref{fin_rem} we describe the limitations  of this article and sketch various possible generalities.

\section{Notation and conventions}\label{not}

We adapt the following notation in this article. The letter $p$ always denotes a prime number, and $\F_p$ the prime field of characteristic $p >0$. By $K$ we mean the field of Laurent-series  $\laurentseries$ over $\F_p$, and $\CalO_K := \powerseries$ its ring of integers with maximal ideal $\maxideal = t \CalO_K$ where $t$ is a fixed uniformizer of $\CalO_K$. Then the group of principal units $ 1 + \maxideal $ is denoted by $U_1$, and its higher unit  subgroups $1 + \maxideal^j$, $j \geq 1$, by $U_j$. 

The Nottingham group over $\F_p$ is always denoted by $\notty$. The notation for strict equivalence is $\simeq$, while that for weak equivalence is $\sim$. Thus $\chi \simeq \psi$ means that $\chi$ and $\psi$ are strictly equivalent, while $\chi \sim \psi$ means that they are weakly equivalent.

The notation ${n \choose k}$ for integers $n, k > 0$ is the usual combinatorial one. More importantly we write ${\alpha \choose k}$ for $\alpha \in \F_p$, to mean any ${n \choose k}$ for $n \equiv \alpha \pmodp.$

In this article, we deal \textit{only} with continuous characters $\chi : U_1 \to \Z/p^n \Z$, with $U_1$ equipped with induced topology from that of topological group $K^\times$, and $\Z/p^n \Z$ with the discrete topology.  So we omit the adjective `continuous' while referring to the characters in the $\zp$-module $\Hom_{\Z_p}^{\textrm{cont}}(\npg, \Z/p^n \Z)$. 

\section{Lubin's framework}\label{Lubin_article}

\newcommand{\normmap}{\mathcal{N}^K_F}
\newcommand{\normresidue}{\rho^K_F}
\newcommand{\boldchi}{\textbf{X}}
\newcommand{\nok}{\mathcal{O}}
\newcommand{\gal}{\textrm{Gal}}

In this section we describe Lubin's  framework to compute torsion elements of order $p^n$ (in the Nottingham group $\notty$); for more details the reader may consult \S 2 of \cite{Lu}.

First, we describe how to associate a torsion element $\gam \in \notty$ of order $p^n$ with a finite abelian extension $K \supset F$ with Galois group isomorphic to $\Z/p^n \Z$ (where, recall, $K = \laurentseries$ as in \S \ref{not}.) Given a torsion element $\gam \in \notty$ of order $p^n$, considered as an element of $\aut(K)$, let $F$ be its fixed field in $K$. Then $K \supset F$ is a finite abelian extension with Galois group $\Gam$ isomorphic to $\Z/p^n \Z$. By the local class-field theory we have an exact sequence 
$$
K^* \xrightarrow{\normmap} F^* \xrightarrow{\normresidue} \Gam \to 1
$$
such that the image of the norm map $\normmap$ is an open subgroup in $F^*$, and for any open subgroup $U$ of finite index in $F^*$ there is a \textit{canonical} abelian extension $L \subset F$ such that $\mathcal{N}^L_F(L^*) = U$. Thus, with the aid of norm residue symbol $\rho^K_F$ we can obtain a character $\boldchi_\gam : F^* \to \Z/p^n \Z$, via the canonical isomorphism $\Gamma \to \Z/p^n \Z$ sending the generator $\gam$ to $1$, and hence a character $\npg \to \Z/p^n \Z$ by restriction; the latter character on $\npg$ is again denoted by $\boldchi_\gam$ abusing notation.

Conversely, given a surjective continuous character $\boldchi: \npg \to \Z/p^n \Z$, we construct a cyclic extension $L \supset K$ and a generator $\gam$ of the  Galois group $\gal(L/K)$ of order $p^n$. Given such a character $\boldchi$, clearly it has a unique extension to $\CalO^*$ and then to $K^*$ by letting to be $0$ at $t$. Thus we obtain a continuous surjective character $\widetilde{\boldchi} : K^* \to \Z/p^n \Z$ with the kernel generated by $t, \ker(\boldchi), $ and $\F_p^*$ (the subgroup $\F_p^*$ is contained in the kernel because its order is $p-1$ which is relatively prime to $p^n$.) The kernel is clearly an open subgroup and hence, by the existence part of local class-field theory, there is a finite abelian extension $L \supset K$, unique up to isomorphism, with Galois group isomorphic to $\Z/p^n \Z$ such that the image of $L^*$ under the norm map $\mathcal{N}^{L}_{K}$ is $\ker(\widetilde{\boldchi})$. Furthermore, we have an exact sequence

\[
\xymatrix{
L^* \ar[r]^-{\mathcal{N}^L_K} & K^* \ar[r]^-{\rho^L_K} \ar[dr]^-{\widetilde{\boldchi}} & \gal(L/K) \ar[r] \ar@{.>}[d]
&1 \\
& & \Z/p^n \Z 
}
\]

The generator $\gam$ of $\gal(L/K)$ is obtained by $\gam := \rho^{L}_K \circ \widetilde{\boldchi}^{-1}(1)$. 

Let us note here that the description of $\gam \in \notty$ in terms of power series $u(t)$ depends on the choice of the uniformizer $t$, but, the conjugacy class of $u(t)$ in $\notty$ does not provided $\mathcal{N}^{L}_K(x) = t$ for a choice of uniformizer $x \in \CalO_L$ (which is used to construct back  $\widetilde{\boldchi}$.) 

\begin{remark}\label{rmk}
\begin{enumerate} 

\item 
Note that the extensions $K/F$ and $L/K$ are totally wildly ramified, whence the terminology `wild' automorphism associated to the elements of Nottingham group $\notty$.

\item 
Let us observe that given a torsion element $\gamma$ the construction yields a \textit{subfield} $F$ of $K$ such that Galois group $\gal(K/F)$ is cyclic of order $p^n$. On the other hand, given a character $\boldchi : U_1 \to \Z/p^n \Z$ the construction gives an \textit{extension} $L$ of $K$ such that $\gal(L/K)$ is cyclic of order $p^n$. 
\item 
In the terminology of Lubin \cite{Lu}, the procedure by which one obtains a torsion series $u(x)$ of order $p^n$ in the Nottingham group from a continuous surjective character $\npg \to \Z/p^n \Z$  is called the \textit{standard procedure}.
\end{enumerate}
\end{remark}

The Nottingham group $\notty$ acts on the group of characters $ \npg \to \Z/p^n \Z$ as follows. For $w \in \notty$ and $\chi$ such a character, one has
$$
{}_w \chi(f(t)) := \chi( f \circ w(t)) 
$$
for $f(t) \in \npg = 1 + t \powerseries$. Hence ${}_w({}_{w'} \chi) = {}_{ww'}\chi$. Now, we impose certain equivalence relation on these group of characters, which, as it turns out, characterizes  conjugacy classes of order $p^n$ elements in $\notty$.

\begin{defn}\label{secdefn}
Two characters $\chi, \psi : \npg \to \Z/p^n \Z$ are said to be \textit{strictly equivalent}, denoted $\chi \simeq \psi$, if there exists an element $w \in \notty$ such that $\psi = {}_w \chi$ and $w(t)/t \in \ker(\chi)$. They are said to \textit{weakly equivalent}, denoted $\chi \sim \psi$, if there exists an element $w \in \notty$ such that $\chi_2 = {}_w \chi_1$.
\end{defn}

As it turns out, strict equivalence is indeed an equivalence relation and is the \textit{right} relation, in contrast to weak equivalence, one needs to impose to characterize conjugacy classes of torsion elements in the Nottingham group: \cite[pp. 5]{Lu}:

\begin{thm}\label{main_thm}
\textrm{\bf (Lubin)} 
If two elements in $\notty$ are conjugate (in that group), then their corresponding characters are strictly equivalent. Conversely, if two characters are strictly equivalent continuous characters of $\npg$ with values in $\Z/p^n \Z$, the torsion power series that arise from them by the \textit{standard procedure} (Remark \ref{rmk}) are conjugate in $\notty$. 
\end{thm}
This result allows one to pass from the analysis of conjugacy classes of torsion elements of order $p^n$, to the analysis of strict equivalence continuous characters $\npg \to \Z/p^n \Z$; the structure of the $\Z_p$-module of continuous characters can be stratified according to the following definition:

\begin{defn}\label{bsdefn}
Let $\chi : \npg \to \Z/p^n \Z$ be a surjective continuous character. The \textit{break sequence} of the character $\chi$ is the tuple $\langle b^{(0)}, \ldots, b^{(n-1)} \rangle$ where for each $j$, the number $b^{(j)}$ is the largest integer $b$ for which there is a $z \in 1 + \nmaxi^b$ such that $\chi(z) = p^{j}$.
\end{defn}

\begin{lem}\label{bs_lemma}
\textrm{\bf (Lubin)} Let $\chi : \npg \to \Z/p^n \Z$ be a surjective continuous character. Let  $\langle b^{(0)}, \ldots, b^{(n-1)} \rangle$ be its break sequence. Then the following conditions hold:
\begin{enumerate}[(a)]
    \item\label{pri_cond} $\gcd(p,b^{(0)}) =1$.
    \item\label{ine_cond} For each $i >0$, $b^{(i)} \geq p b^{(i-1)}$, and
    \item\label{str_cond} If the above inequality is strict, then $\gcd(p,b^{(i)}) =1$.
\end{enumerate}
Conversely, every sequence $\langle b^{(0)}, \ldots, b^{(n-1)} \rangle$ satisfying the above three conditions is the break sequence of some character $\chi$ on $\npg$. There are only finitely many different characters $\chi$ with the break sequence $\langle b^{(0)}, \ldots, b^{(n-1)} \rangle$, and a \textrm{fortiori}, only finitely many strict equivalence classes of such characters.
\end{lem}

\newcommand{\tr}{\textrm{Tr}}
\noindent Lubin classified torsion elements of order $p^2$ of break sequence type $\langle 1, m \rangle$ over a finite field $\kappa = \F_{p^f}$. We discuss it briefly. The principal unit group $U_1$ is endowed with the structure of a $W_\infty(\kappa)$-module with the aid of Artin-Hasse exponential, where $W_\infty(\kappa)$ is ring of Witt vectors over $\kappa$. Then $\Hom_{\Z_p}^{\textrm{cont}}(U_1, \Z/p^2 \Z)$ has the structure of a free $W_2(\kappa)$-module, with basis $\Set{\mathfrak{Z}_m}$, one for each $ m$ with $p \nmid m$.  

Let $\chi$ be a character of order $p^2$ and type $\langle 1, m \rangle$. As it turns out, the character $\chi$ can be expressed as 
$$
\chi = a_0 \mathfrak{Z}_1 + \sum_{\substack{j \geq 1 \\ p \nmid j}} a_j p \mathfrak{Z}_j.
$$
where $a_j \in \kappa$, for all $j \geq 0$, and in particular $a_0, a_m \in \kappa^*$.
Let $\tr: \kappa \to \F_p$ denote the (usual) trace map.  Define the \textit{indicator} $\ind \chi$ as
$$
\ind(\chi) =
\begin{cases}
 [a_0, \tr(a_{p-1}/a_0^{p-1})] & \textrm{ if } m = p\\
 [a_0,a_m, 0] & \textrm{ if } m \equiv 1 \pmod{p}\\
[a_0,a_m, \tr(a_0a_{m-1}/ a_m)] & \textrm{ otherwise}
\end{cases}
$$

\begin{thm}\label{Lubin_<1,m>}
\textrm{\bf (Lubin)}\footnote{Unpublished work of Lubin}   Let $\chi$ and $\psi$ be characters of type $\langle 1,m \rangle$ over a finite field $\kappa$ of characteristic $p$.
Then $\chi$ and $\psi$ are strictly equivalent if and only if $\ind \chi = \ind \psi$.
\end{thm}

The main object of this article is to classify torsion elements of the Nottingham group over the prime field $\F_p$, of order $p^2$ and type $\langle 2, m \rangle$.

\section{Topological structure of $U_1$ and consequences}
\label{top_str}

In this section, we describe the topological structure of $\npg$, then analyze the structure of the $\Z_p$-module $\Hom_{\Z_p}(\npg, \zmodpnz)$ in which the characters $\chi: \npg \to \zmodpnz$ live. We obtain a coarse classification of the equivalence classes of strict equivalent characters which will be refined in the next section. 

\subsection{Topological structure of $\npg$}

Let us begin by recalling the topological structure of the multiplicative group $\npg$ \cite[Chap. 2]{Iw}. The principal unit group $\npg$ has a $\Z_p$-module structure as follows. More generally any finite abelian $p$-group  has a structure of a $\Z/p^n \Z$- module for any $n$ sufficiently large. Indeed, the canonical isomorphism $\zmodpnz \cong \Z_p/p^n \Z_p$ together with the natural projection map $\Z_p \to \Z_p/p^n \Z_p$ gives a canonical structure of $\Z_p$-module. Hence any abelian pro-$p$-group $G$ has also a canonical $\Z_p$-module structure; let us recall any abelian pro-$p$-group is the projective limit of a family $\Set{G_i}$ of finite abelian $p$-groups. In particular, since $\npg$ is the projective limit of finite abelian $p$-groups $\npg / \npgpower{n}$ with respect to the canonical maps $\npg / \npgpower{m} \to \npg / \npgpower{n}$ for $m \geq n \geq 0$, we have :
$$
\npg = \varprojlim_{n} \npg/ \npgpower{n}. 
$$
\noindent Alternatively, we have the following topological isomorphism \cite[Prop. 2.8]{Iw}
\begin{prop}\label{u1_structure}
The $\Z_p$-module $\npg$ is topologically isomorphic to the direct product of countably infinite copies of $\Z_p$. In fact,
$$
\npg \cong \prod_{\substack{ n \geq 1 \\ p \nmid n}} (1+t^n)^{\Z_p}
$$
\end{prop}
Let us note that the family $\Set{1 + t^m }$ is a topological basis of $\npg$. This basis serves well while working over a prime field of characteristic $p >0$, and not so well over arbitrary finite fields. In the latter case the basis constructed with the aid of Artin-Hasse exponential is almost mandatory; the reader may refer to \cite[\S 17.5]{Ha} for a full treatment of the topic. 
\subsection{Description of the module $\Hom_{\Z_p}^{\textrm{cont}}(\npg, \Z/p^n \Z)$}\label{desc}
The principal unit group $\npg$ is equipped with two actions, namely that of Nottingham group $\notty$ and that of $\Z_p$. The action of Nottingham group $\notty$ via $(z,u) \mapsto z \circ u$ (where the $\circ$ denotes composition of power series) for $z \in \npg$ and $u \in \notty$. (Recall, that if $ u = t( 1 + \sum_{i \geq 1} a_i t^i)$, and $z = 1 + \sum_{j \geq 1} b_j t^j$ then $z \circ u := 1 + \sum_{j \geq 1} b_j u^j$.) On the other hand, by means of exponentiation by a $p$-adic integer we have the natural action of $\Z_p$ on $\npg$; this action is well-defined since for $z \in \npg$ the $(t)$-adic limit of the sequence $\Set{z^{p^n}}$ is $1$.
These two actions on $\npg$ are compatible, i.e.  for $a \in \zp, u \in \notty$, we have
$$
(z^a) \circ u = (z \circ u)^a,
$$
which can be readily verified by replacing $a$ with approximation by integers $(a_n)$ for $a_n \in \zmodpnz$ converging to $a$. 

From the description of $U_1$ as a $\Z_p$-module, it follows that we have a $\Z/p^n \Z$-module structure on $\Hom_{\zp}^{\textrm{cont}}(\npg, \zmodpnz)$ with basis $\FZ_m$ dual to the basis $\Set{1 + t^m }$ of $\npg$, where $m \geq 1$ and $(m,p) =1$: with the notation $E_j := 1 + t^j$, we have $\FZ_i( E_j) = \delta_{ij}$  where $\delta_{ij}$ is the Kronecker delta function. In other words, the set of continuous characters from $U_1 \to \Z/p^n \Z$ is the direct sum (not product) of the character groups on the factors $\langle 1 + t^m \rangle$ into $\Z/p^n \Z$ for $p \nmid m$.  It follows that a character $\chi \in \Hom_{\zp}(\npg, \zmodpnz)$ has the following form
$$
\chi = \sum_{ \substack{j \geq 1 \\ p \nmid j}} a_j \FZ_j \; \; a_j \in \Z/p^n \Z.
$$
Now, let $\chi : \npg \to \zmodpnz$ be a surjective (continuous) character with break sequence $\langle b^{(0)}, \ldots, b^{(n-1)} \rangle$. From the definition of the break sequence [Def. \ref{bsdefn}] it follows that the kernel  of $\chi$  contains $\npgpower{b^{(n-1)}+1}$. Therefore all characters with this break sequence are in $\Hom_{\Z_p}(\npg/\npgpower{b^{(n-1)}+1}, \zmodpnz)$, which is a direct sum of cyclic $\Z/p^n \Z$-modules, one for each integer $j \leq m$ prime to $p$.

\subsection{Strict equivalence condition}

\begin{lem}\label{lubin_lemma} \footnote{I thank Jonathan Lubin for sharing this proof with me.}
Let $\chi, \psi : \npg \to \zmodpnz$ be two continuous characters. Then $\chi$ is strictly equivalent to $\psi$ if and only if there an $u \in \notty$ such that $\psi = {}_u \chi$ and $\chi( u(t)/t ) \equiv 0 \pmod{p^{n-1}}$.
\end{lem}

\begin{proof}
Necessity of the condition follows from the definition: if $u(t)/t \in \ker(\chi)$, i.e. $\chi(u(t)/t) \equiv 0 \pmod{p^n}$ then obviously $\chi(u(t)/t) \equiv 0\pmod{p^{n-1}}$. As for the sufficiency of the condition, let $\genbs$ be the break sequence of the character $\chi$. Let $m = b^{(n-1)}$. By definition $\npgpower{m+1} \subset \ker(\chi)$ and $U_m \not \subseteq \ker(\chi)$. We claim that if $z \in \npgpower{m}$ and $w = t z \in \notty$ then ${}_w \chi = \chi$. Indeed, evaluating at the basis elements $E_j := 1 + t^j$, and noting that 
\begin{align*}
1 + w^j &= 1 + t^j z^j  = 1 + t^j (1 + t^m z')  \\
&= 1 + t^j + t^{j+m}z' \equiv  1 +t^j \pmod{U_{m+1}}
\end{align*} 
for some $z' \in \F_p[\![t ]\!]$, we obtain 
\[
{}_w \chi( E_j) = \chi( E_j \circ w ) = \chi(1 + w^j) 
                 \equiv \chi(1+t^j)  \equiv \chi(E_j)\pmod{\npgpower{m+1}}.
\]
Let $u \in \notty$ be the element given by the condition, i.e. that $\psi = {}_u \chi$ and $\chi(u(t)/t)  \equiv 0 \pmod{p^{n-1}}$. This implies, in particular, that $p^{-(n-1)} \chi(u(t)/t)) = \lambda \in  \F_p$. More generally, we have a nonzero $\F_p$-linear map
$$
\widetilde{\chi} : \F_p \to \F_p
$$
defined by $\widetilde{\chi}(a) = \chi(1+ a t^m)/p^{n-1} \in \F_p$.  In particular, there exists $a \in \F_p$ such that $\widetilde{\chi}(a) = \chi(1 + at^m)/p^{n-1} = - \lambda \in F_p$.

Put $w(t) = t + a t^{m+1} \in \notty$. Then ${}_{u \circ w} \chi = {}_u ({}_w \chi) = {}_ u \chi = \psi$, so that the first  condition still holds. We now calculate $\chi(u \circ w(t))/t$. Since $(u \circ w)(t) \equiv u(t) + a t^{m+1} \pmod{(t^{m+2})}$, we have
\begin{align*}
(u \circ w)(t)/t &\equiv \frac{u(t)}{t} + at^m \pmod{t^{m+1}} \\ 
&\equiv \frac{1}{t}u(t) (1 + a t^m) \pmod{t^{m+1}},
\end{align*}
and therefore
\[
 \chi ( (u \circ w)(t)/t )/p^{n-1} = \chi(u(t)/t)/p^{n-1} + \chi(1 + a t^m) = \lambda - \lambda =0.
\]
Thus, $\chi ( (u \circ w)(t)/t )/p^{n-1}$ is divisible by $p$, hence $\chi ( (u \circ w)(t)/t ) \equiv 0 \pmod{p^n}$ so that, by definition, $\psi$ is strictly equivalent to $\chi$.
\end{proof}
\section{Structure of characters}\label{str_char}
Now we begin the classification of strict equivalent classes of type $\langle 2, m \rangle$ corresponding to Nottingham elements of order $p^2$. Up to conjugacy in $\notty$, they correspond to strict equivalence classes of characters of type $\langle b^{(0)}, b^{(1)} \rangle$ [Theorem \ref{main_thm}]. From Lemma \ref{bs_lemma} it follows that $b^{(0)}$ is prime to $p$, and  $b^{(1)} \geq p b^{(0)}$ which, if strict, implies that $b^{(1)}$ is prime to $p$, otherwise $b^{(1)} = 2 p$. In this article, we classify strict equivalence classes where $b^{(0)}=2$. Then since $p \nmid b^{(0)}$, the prime $p \neq 2$. The classification depends on whether $b^{(1)}$ is prime to $p$ or not, and if it is prime to $p$ then whether it is congruent to $1 \pmod{p}$ or not.  Note that there is only one case for $m \equiv 0 \pmodp$, namely that $m = 2p$; thus we have only `exceptional' case $\langle 2, 2p \rangle$. Also note that a character of certain type $\langle b^{(0)}, b^{(1)} \rangle$ is automatically surjective by definition. So when we talk of characters `of type\ldots' we do not explicitly mention that the character is surjective. Recall that further  we omit the adjective `continuous', as no other kind is considered in this article; see \S \ref{not}.

Consider the $\Z/p^2 \Z$-module $\Hom_{\Z_p}^{\textrm{cont}}(U_1, \Z/p^2 \Z)$ of characters, $\chi : 1 + \FM \to \Z/p^2 \Z$. For $p \nmid j$, let  $\FZ_j$ be the canonical basis dual to $E_j := 1 + t^j$, i.e. $\FZ_j( E_i) = \delta^{i}_{j}$, the Kronecker delta function.  Let $\CalA = \{ 0, 1 \ldots, p^2 -1 \}$ be a fixed choice of representatives of $\Z/p^2 \Z$. Then any character $\chi$ of type $\langle 2,m \rangle$ has the expression of the form
$$
\chi = \sum_{ \substack{1 \leq j \leq m \\ p \nmid j}} a_j \FZ_j.
$$ 
with the coefficients $a_j \in \CalA$. From the definition of break sequence [Definition \ref{bsdefn}], it follows that
\begin{enumerate}\label{coe_pro}
\item
$a_1 \in \Z/p^2 \Z$,
\item
$a_2 \in (\Z/p^2 \Z)^*$,
\item 
For all $3 \leq j \leq m$ such that $p \nmid j$, the coefficient  $a_j \in \Set{0,p,2p, \ldots, (p-1)p}$, 
\item
If $(m,p) =1$, then $a_m \neq 0$.
\end{enumerate}
\noindent Writing $a_1,a_2$ in the form $x + p . y$ where $x,y \in \{ 0, 1 \ldots, p-1 \}$, and changing the notation, the above expansion takes the following form
$$
\chi = x_1 \FZ_1 + x_2 \FZ_2 + \sum_{\substack{1 \leq j \leq m \\ p \nmid j}} a_j. p \FZ_j,
$$
where now $x_1, x_2, a_1, \ldots, a_m   \in \{ 0, 1, \ldots p-1 \}$, and $x_2 \neq 0$, and if $(m,p) =1 $ then $a_m \neq 0$ too. To emphasize further, the new $a_j$ after abuse of notation is equal to the `old' $a_j$ divided by $p$. We call this expansion of $\chi$ as in equation \textit{standard expansion}.  

Let $\chi, \psi$ be two characters of type $\langle 2, m \rangle$ with standard expansions:
\begin{equation}\label{sta_exp}
\begin{split}
\chi = x_1 \FZ_1 + x_2 \FZ_2 + \sum_{\substack{1 \leq j \leq m \\ p \nmid j}} a_j. p \FZ_j, \\
\psi = y_1 \FZ_1 + y_2 \FZ_2 + \sum_{\substack{1 \leq j \leq m \\ p \nmid j}} b_j. p \FZ_j.
\end{split}
\end{equation} The following lemma will be used throughout the article; its motivation will be apparent while proving the main result.

\begin{lem}\label{coarse}
Let $\chi, \psi$ be two characters of type $\langle 2, m \rangle$ with standard expansions as in equation \eqref{sta_exp}. Let $u(t) = t(1 + \alpha t + \beta t^2 \ldots) \in \notty$. Then,

\begin{enumerate}[(a)]
\item\label{12}
$\ub{u} \chi(E_1(t)) \equiv x_1 + \alpha x_2 \pmodp$,  and $\ub{u} \chi(E_2(t)) \equiv x_2 \pmodp$.
\item\label{m}
If $p \nmid m$ then $\ub{u} \chi(E_m(t)) \equiv p. a_m \pmodpp$.

\item\label{m-1}
If $m \neq 0,1 \pmodp$ then $\ub{u} \chi(E_{m-1}) \equiv p (a_{m-1} + (m-1) \alpha a_m) \pmodpp$.

\item\label{meq_0}
If $m \equiv 0 \pmod p$, so that $m = 2p$, then 
$\ub{u} \chi(E_{m-1}) \equiv p (a_{m-1} + \eta x_2) \pmodpp$ where $\eta^p = (m-1)\alpha$.

\end{enumerate}

\end{lem}

\begin{proof}

\begin{enumerate}[(a)]
\item
Since $E_1 \circ u(t) = 1 + u(t) = 1 + t(1 + \alpha t + \beta t^2 + \ldots) = 1 + t + \alpha t^2 + \beta t^3 + \ldots  = (1 + t) (1 + t^2)^{\alpha} \ldots$ we have, modulo $p$,
\begin{align*}
\ub{u} \chi (E_1(t))
&= \chi(E_1 \circ u(t)) = \chi((1+t) (1 + t^2)^\alpha) \\
&=  \chi(E_1(t)) + \chi(E_2^\alpha (t) = x_1 + \alpha x_2. 
\end{align*}
\noindent Similarly, the other relation is obtained by evaluating at $E_2$, and using the expansion $1 + u(t)^2 = 1 + t^2 (1 + \alpha t + \beta t^2 + \ldots)^2 = 1 + t^2 + \ldots = (1 + t^2) \ldots$. Thus, modulo $p$, 
\[
{}_{u} \chi (E_2(t) ) = \chi(E_2  \circ u(t)) = \chi (1 + t^2) = x_2.
\]

\item
 If $p \nmid m$ then, modulo $p^2$,
\begin{align*}
{}_{u} \chi (E_m(t) )  =  \chi( E_m \circ u(t)) =  \chi (1 +  t^m) = \chi(E_m(t)) = pa_m  .
\end{align*}

\item
First note that $ E_{m-1} \circ u(t) = 1 + u(t)^{m-1} =
1 + t^{m-1} + (m-1) \alpha t^m + \ldots$, then the conclusion follows by evaluating $\ub{u} \chi$  at $E_{m-1}$.

\item
First observe that $1 + \alpha t^m = (1 + \eta t^2)^{p}$ with $\eta^p = \al$. Then we have 
\begin{align*} 
E_{m-1} \circ u(t) &= 1 + t^{m-1} + (m-1) \al t^m + \ldots \\
&\equiv  (1 + t^{m-1}) (1 + (m-1)\alpha t^m)  \pmod{U_{m+1}} \\
&\equiv (1 + t^{m-1})(1 + \eta t^2)^p  \; \; \textrm{ where } \eta^p = (m-1) \alpha \pmod{U_{m+1}}\\
&\equiv E_{m-1} E_2^{\eta p} \pmod{U_{m+1}}, 
\end{align*}
and therefore, modulo $p^2$,
\begin{align*}
\ub{u} \chi(E_{m-1}(t)) &= p. a_{m-1} + p. \eta( x_2 + p. a_2) \\
&= p (a_{m-1} + \eta x_2).
\end{align*}
\end{enumerate}
\end{proof}
\noindent As noted in the introduction, only some of the coefficients in the standard expansion of characters together with a relation among them,  determine the character up to strict equivalence.  Accordingly we collect these coefficients into \textit{indicator} of the character $\chi$: 
\begin{defn}\label{ind_defn}
Let $\chi$ be a character of type $\langle 2, m \rangle$, and $ \chi = x_1 \FZ_1 + x_2 \FZ_2 + \sum_{\substack{1 \leq j \leq m \\ p \nmid j}} a_j. p \FZ_j$  be its standard expansion . The \textit{indicator} of $\chi$, denoted $\ind \chi$, is defined as follows:
\begin{enumerate}[(i)]
\item
If $m \equiv 0 \pmodp$,  so that $m = 2p$, then
$$
\ind \chi = [x_2,\frac{x_1}{x_2} -  \frac{ a_{m-1}^p}{(m-1)x_2^p}].
$$
\item
If $m \not \equiv 0,1 \pmodp$, then 
$$
\ind \chi = [ x_2, a_m, \frac{x_1}{x_2} -  \frac{ a_{m-1}}{(m-1)a_m}].
$$
\item
If $m \equiv 1 \pmodp$, then 
$$
\ind \chi = [x_2, a_m].
$$
\end{enumerate}
\end{defn}

Let us denote the last coordinate of $\ind \chi$ by $\ind_0 \chi$.

\begin{lem}\label{nec_lem}
Let $\chi, \psi$ be two characters of type $\langle 2, m \rangle$ with	 standard expansions as in equation \eqref{sta_exp}.
Suppose $\psi = \ub{u} \chi$ for some $u(t) = t(1 + \alpha t + \beta t^2 \ldots) \in \notty$.. Then $\ind_0 \ub{u} \chi = \ind_0 \psi$.
\end{lem}

\begin{proof}
There is nothing to show in the case where  $m \equiv 1 \pmodp$. Consider the case where $m \not \equiv 0,1 \pmod{p}$. Since 
\begin{align*}
E_{m-1} \circ u(t) &= 1 + t^{m-1}(1 + \alpha t + \beta t^2 + \ldots ) =  1 + t^{m-1} + \alpha t^m + \beta t^{m+1}+ \ldots \\
&= (1 + t^{m-1})(1 + t^m)^{\alpha} \ldots	= E_{m-1}(t) E_m^{\alpha}(t) 
\end{align*}
we see that, modulo $p^2$,
\begin{align*}
\psi(E_{m-1}(t)) &= \ub{u} \chi(E_{m-1}(t)) \\
&= \chi( E_{m-1} \circ u(t)) = \chi( E_{m-1}(t) E_m^\alpha(t)) \\
&= p. a_{m-1} + p \alpha a_m, 
\end{align*}
and hence $\alpha = (b_{m-1} - a_{m-1})/(m-1) a_m$. From the relations \eqref{12}, \eqref{m} of Lemma \ref{coarse}, it follows that $y_2 = x_2$, $a_m = b_m, y_1 = x_1 + \alpha x_2$, and consequently  we have
\begin{align*}
\frac{x_1}{x_2} - \frac{a_{m-1}}{(m-1)a_m} &=  \frac{x_1 - y_1}{x_2} + \frac{y_1}{x_2} - \frac{b_{m-1}}{(m-1)a_m} + \frac{b_{m-1} - a_{m-1}}{(m-1)a_m}  \\
&= - \alpha + \frac{y_1}{y_2} - \frac{b_{m-1}}{(m-1)a_m} + \alpha =\frac{y_1}{y_2} - \frac{b_{m-1}}{(m-1)b_m}.
\end{align*}

Now consider the case $m \equiv 0 \pmod{p}$. From Lemma \ref{coarse} \eqref{meq_0}, and the hypothesis that $\psi = {}_u \chi$, it follows that $b_{m-1} = a_{m-1} + \eta x_2$ where $ \eta = (b_{m-1}-a_{m-1})/x_2$ and $\eta^p = (m-1) \alpha$ so that $\alpha = (b_{m-1}^p - a_{m-1}^p)/(m-1) x_2^p$. As previously, $y_1 = x_1 + \alpha x_2$. Equating these expressions for $\alpha$ and rearranging the terms we obtain the desired conclusion. 
\end{proof}

\section{Classification of weak equivalence classes}
\label{main_section}

In this section we prove the main result of the article. Let us fix characters $\chi, \psi$ of type $\langle 2, m \rangle$ with standard expansions
\begin{equation}\label{exp}
\chi = x_1 \FZ_1 + x_2 \FZ_2 + \sum_{\substack{1 \leq j \leq m \\ p \nmid j}} a_j. p \FZ_j, \; \; 
\psi = y_1 \FZ_1 + y_2 \FZ_2 + \sum_{\substack{1 \leq j \leq m \\ p \nmid j}} b_j. p \FZ_j.
\end{equation}

\begin{thm}\label{my_main_thm}
Let $\chi, \psi$ be characters of type $\langle 2, m \rangle$.  Then $\chi \sim \psi$ if and only if  $\ind \psi = \ind \chi$.
\end{thm}
\begin{proof}

Let $\chi, \psi$ have standard expansions as in equation \eqref{exp}. Necessity of the condition follows from the already established results. Indeed, from Lemma \ref{coarse} \eqref{12} it follows that $a_2 = b_2$ in all cases and  $a_m = b_m$ if $m \not \equiv 0 \pmodp$. The remaining equality that $\ind_0 \chi = \ind_0 \psi$ follows from  Lemma \ref{nec_lem}. 

The hard part is to prove the sufficiency of the condition, namely that $\ind \chi = \ind \psi$  implies $\chi \sim \psi$. First observe that Lemma \ref{coarse} can be applied with $u = t(1 + \alpha t + \beta t^2)$, for some $\alpha, \beta \in \F_p$,  since terms of degree $\geq 4$ in $t$ do not affect the conclusions of the lemma. The strategy of the proof is to use backwards induction, beginning with index $m-1$ in the case where $m \not \equiv 0,1 \pmodp$, with index $m-2$ in the case where $m \equiv 1 \pmodp$, and with index $m-1$ in the case where $m \equiv 0 \pmodp$.

First consider the case where $m \not \equiv 0,1 \pmodp$. By hypothesis $\ind \chi = \ind \psi$, hence in particular $a_m = b_m$. Let $u = t(1 + \alpha t)$ where $\alpha = (b_{m-1}- a_{m-1})/(m-1)a_m$. Then the coefficients of $\ub{u} \chi,\psi$ with index $m-1$, resp. $m$ are equal.  Moreover since $\ind_0 \chi = \ind_0 \psi$, we obtain $$
\alpha  = \frac{y_1 - x_1}{x_2} = \frac{b_{m-1} - a_{m-1}}{(m-1)a_m}.
$$
Consequently by Lemma \ref{coarse}\eqref{12},  $\ub{u} \chi (E_1) = y_1 = \psi(E_1)$ and $\ub{u} \chi(E_2) = y_2 = \psi(E_2)$. Thus, together with the hypothesis that $x_2 = y_2$, we have
$$
\ub{u} \chi = y_1 \FZ_1 + y_2 \FZ_2 + \sum_{ 1 \leq j \leq m-2, p \nmid j} a_j p \FZ_j + b_{m-1} p \FZ_{m-2} + b_{m} p \FZ_m.
$$

We now assume, inductively, that there is an integer $N$, where  $3 < N < m-1$, such that for all $j$, $N \leq j < m $, we have $b_j = a_j$.
 To complete the induction then, we need to construct $w \in \notty$ such that the coefficients of $\ub{w} \chi $ and $\psi$ with index $ 3 \leq j-1 \leq k \leq m, p \nmid k$ are equal. If $p \mid (j-1)$ then the coefficient with index $j-1$ does not occur in the standard expansion of $\chi$, so then there is nothing to prove. Otherwise, let $w = t (1 + \alpha t^{m-(j-1)})$ with $\alpha = (b_{j-1} - a_{j-1})/(j-1) a_m$, then clearly $b_{j-1} = \psi(E_{j-1}(t)) = (\ub{w} \chi)(E_{j-1}(t))$, and the $b_k = a_k$ for all $j-1 \leq k < m $, $p \nmid k$. The induction is complete, and we may assume $\chi$ to be of the following form
$$
\chi = y_1 \FZ_1 + y_2 \FZ_2 +  a_1' p \FZ_1 + a_2' p \FZ_2 + \sum_{ 3 \leq j \leq m, p \nmid j} b_j p \FZ_j 
$$
for some $a_1', a_2' \in \F_p$.  Note that at each stage of the induction after the base case the coefficients $y_1, y_2$  do not change. Furthermore, it is evident that indicator of the characters does not change either.

Now it remains to deal with the coefficients $a_1', a_2'$. Let $u = t(1 + \alpha t^{m-2})$ with $\alpha = (b_2 - a_2')/2 b_m$. Then,  
\begin{align*}
\ub{u} \chi(E_2(t)) &=  \chi(E_2 \circ u(t)) = \chi(1 + t^2(1 + 2 \al t^{m-2})) \\
&= \chi(E_2 . E_m^{2 \al}) = y_2 + p (a_2' + 2 \al b_m) \\
&= y_2 + p b_2 = \psi(E_2(t))
\end{align*}
It remains to consider the coefficient with index $1$ case.   Let $u = t( 1 + \alpha t^{m-1})$ where $\alpha = (b_1 - a_1^{''})/b_m$, where $a_1^{''}$ is the altered coefficient $a_1'$ in $\ub{u} \chi$, where $u$ is the element chosen in the previous step. Then, modulo $p^2$,
\begin{align*}
\ub{u} \chi (E_1 (t) ) &= \chi((1 + t)(1 + t^m)^\alpha) = y_1 + pa_1' + p \alpha b_m \\
&= y_1 + p \Big( a_1 + \frac{b_1 - a_1'}{b_m} b_m \Big) \\
&= y_1 + p. b_1 = \psi(E_1(t)). 
\end{align*}
This completes the proof in the case where $m \not \equiv 0, 1 \pmodp$.

Now consider the case where $m \equiv 1 \pmodp$. Then the basis of induction begins at $m-2$, and we let $u = t(1 + \beta t^2)$ where $\beta = (b_{m-2}-a_{m-2})/ (m-2) a_m$. Consider the composite $U = u \circ (1 + \alpha t)$ where $\alpha = (y_1 - x_1)/x_2$. Then $\ub{U} \chi$ has the following form
$$
\ub{U}\chi = y_1 \FZ_1 + y_2 \FZ_2  + \sum_{ 1 \leq j \leq m-1, p \nmid j} a_j p\FZ_j + b_m p \FZ_m. 
$$
Now we proceed as in the previous case to complete the induction. 

Finally consider the case where $m \equiv 0 \pmodp$. As already noted, we have $m = 2p$. In this case, the basis of induction begins at index $m-1$, and by Lemma \ref{coarse} \eqref{meq_0} the basis step holds for appropriate $\alpha$, and the induction proceeds as in the $m \not \equiv 0,1 \pmodp$ case. This finishes the proof the sufficiency of the condition $\ind \chi = \ind \psi$ to imply that $\chi$ is weakly equivalent to $\psi$.
\end{proof}

\begin{cor}\label{weak_corollary}
Let $d_m^{\textrm{weak}}$ be the number of weak equivalence classes of characters of type $\langle 2, m \rangle$. Then,
\[ 
d_m^{\textrm{weak}} = 
\begin{cases}
p(p-1)	 & \; \textrm{if} \; m \equiv 0 \pmodp. \\
(p-1)^2	 & \; \textrm{if}   \; m \equiv 1 \pmodp.\\
p(p-1)^2	 & \; \textrm{otherwise}. \\
\end{cases}
\]  
\end{cor}

\begin{proof}
Recall that the only break sequence with $p \mid m$ is $\langle 2, 2p \rangle$. The corollary follows immediately from Definition \ref{ind_defn}, Theorem \ref{my_main_thm}, and the properties of the coefficients $x_1, x_2$ and $a_j$ for $1 \leq j \leq  m$, $p \nmid j$  [refer to the discussion before Lemma \ref{coarse}], namely that $x_2, a_m \in \{ 1, \ldots, p-1 \}$, and $x_1 \in \{0, \ldots, p-1 \}$: these coordinates of the indicator can be chosen independent of each other.
\end{proof}

\section{Classification of strict equivalence classes}\label{classification_strict_classes}

\noindent Throughout the section, we consider characters $\chi, \psi$ of type $\langle 2, m \rangle$ with standard expansions as in equation \eqref{exp}. Consider the following subset of $\notty$: 
\begin{equation*}\label{x1x2_set}
\notty(x_1,x_2) := \Set{ u = t(1 + \alpha t + \beta t^2 + \ldots) \in \notty | \frac{x_1}{x_2} \alpha - {\alpha \choose 2} + \beta  \equiv 0 \pmodp.}
\end{equation*}
\noindent Note that the set is well-defined because $x_2 \in \F_p$ is nonzero.
\begin{prop}\label{new_addition}
The set $\F_p \times \F_p$ equipped with the operation $\oplus$ defined by
$$
(a,b) \oplus (c,d) = (a+c, b+ d + 2 ac)
$$
is a finite abelian group, where the $+$ on the right hand side of the equality is the usual addition in $\F_p$. The map $\Phi : \notty \to (\F_p \times \F_p, \oplus)$ defined by
$$
u = t(1 + \alpha t + \beta t^2 + \ldots) \mapsto  (\alpha, \beta)
$$ is a surjective group homomorphism. Furthermore,
\begin{enumerate}[(i)]
  \item \label{conjugate_invariance}
$\notty(x_1,x_2)$  contains $\Ker \Phi$, and is invariant under conjugation in $\notty$.
\item \label{notty_decomposition}
 $\notty = \mathop{\bigcup}_{k=0}^{p-1} \notty(x_1,x_2) g(k)$, a disjoint union.
\end{enumerate}
\end{prop}

\newcommand{\ol}[1]{\overline{#1}}
\begin{proof}
It is straightforward to verify that $(F_p \times F_p, \oplus)$ is an abelian group; we just note that the identity element is $(0,0)$ and the inverse of $(a,b)$ is $(-a, -b + 2 a^2)$. We now show that $\Phi$ is a group homomorphism. Clearly the identity element $t \in \notty$ is mapped to the identity element $(0,0)$ of $(\F_p \times \F_p, \oplus)$. Let  $u = t(1 + a t + bt^2 + \ldots ), v = t(1 + ct + d t^2 + \ldots) \in \notty$. Then  $u \circ v = t( 1 + (a + c) t + (b + d + 2 ac)t^2 + \ldots )$, so that 
$$
\Phi(u \circ v) = (a+c, b+d + 2 ac) = (a,b) \oplus (c,d) = \Phi(u) \oplus \Phi(v).
$$
(In fact, the composition $\oplus$ on $\F_p \times \F_p$ is defined in such a way that $\Phi$ turns out to be a group homomorphism.) The remaining assertions, except perhaps \eqref{notty_decomposition}, are straightforward and may be seen immediately with the aid of the map $\Phi$ together with the fact that the group $\notty/\Ker \Phi \cong (\F_p \times \F_p, \oplus)$ is abelian. 

We now prove assertion \eqref{notty_decomposition}. First, observe that for $u = t(1 + a t + b t^2 + \ldots) \in \notty(x_1,x_2)$, if $a \equiv 0 \pmodp$, then $b \equiv 0 \pmodp$ so that $u \in \ker(\Phi)$. In other words, if $a = 0$ and $b \neq 0$ then $u \not \in \notty(x_1,x_2)$. Therefore, we consider  elements $g(k) := t(1 + kt^2) \in \notty \setminus \notty(x_1,x_2)$, for all $1 \leq k \leq p-1$. Now, for $u\in \notty$ as above, define $c = \frac{x_1}{x_2} a - {a \choose 2} + b \in \F_p$. Then
$$
\Phi(  u \circ g(c)^{-1}) =  \Phi(u) \oplus \Phi(g(c))^{-1}
=  (a,b) \oplus (0,-c) = (a, b-c),
$$  
and since $ \Big( \frac{x_1}{x_2} a - {a \choose 2} \Big)+ (b-c)  = \Big( \frac{x_1}{x_2}a + b - {a \choose 2} \Big) - c=  0$, $ u  \circ g(c)^{-1} \in \notty(x_1,x_2)$, whence $u \in \notty(x_1,x_2) g(c)$. It remains to show that the union is disjoint. Indeed, suppose  $u \circ g(k) = v \circ g(l)$, for some $u = t(1 + at +bt^2 + \ldots)$,  $v = t(1 + ct + dt^2 + \ldots) \in \notty(x_1,x_2)$, where $0 \leq k, l \leq p-1$. Then $\Phi(u) \oplus \Phi(g(k)) =  \Phi(v) \oplus \Phi(g(l))$, hence $(0,k) \oplus (a,b) = (0,l) \oplus (c,d)$, consequently $(a, k-l + b) = (c,d)$, whence $a = c$ and $k-l = d-b$. Since $v \in \notty(x_1,x_2)$ we have $\frac{x_1}{x_2} c - {c \choose 2} + d \equiv 0 \pmodp$, equivalently $\big( \frac{x_1}{x_2} a - {a \choose 2} + b \big) + k-l \equiv 0 \pmodp$ since $u \in \notty(x_1,x_2)$,  hence $k = l$.
\end{proof}

\begin{lem}\label{second_condition_lemma}
Suppose $\ind \chi = \ind \psi$, so that $\chi \sim \psi$ by Theorem \ref{my_main_thm}, i.e. $\psi = \ub{u} \chi$ for some $u \in \notty$. Suppose among \textit{such} $u \in \notty$, there exists $u \in \notty(x_1,x_2)$. Then  $\chi \simeq \psi$. Conversely, if $\chi \simeq \psi$,  there exists $u \in \notty(x_1,x_2)$ such that $\psi = \ub{u} \chi$ (and therefore $\ind \chi = \ind \psi$ by Theorem \ref{my_main_thm}).
\end{lem}

\begin{proof}
By Definition \ref{secdefn} and Lemma \ref{lubin_lemma}, it follows that $\chi \simeq \psi$ if and only if there exists $u \in \notty$ such that $\psi = \ub{u} \chi$ and $\chi(u(t)/t) \equiv 0 \pmodp$. Both the implications of the lemma follow, based on their respective hypothesis, if we show that the condition $\chi(u(t)/t) \equiv 0 \pmod{p}$ is \textit{equivalent} to the condition $\frac{x_1}{x_2} \alpha  + \beta - {\alpha \choose 2 }) \equiv 0 \pmodp$. But this is immediate, since $u(t)/t = 1 + \alpha t + \beta t^2 + \ldots = (1 + t)^\al ( 1 + t^2)^{\beta - {\alpha \choose 2}} \ldots$ so that 
$
\chi(u(t)/t) \equiv  x_1 \alpha + x_2 \Big(\beta - {\alpha \choose 2}\Big) \pmodp.
$
Since $x_2 \in \F_p$ is nonzero, the desired equivalence follows.
\end{proof}

\begin{thm}\label{final_thm}
Let $d_m$ denote the number of conjugacy classes of Nottingham elements of order $p^2$ and type $\langle 2, m \rangle$. Then
\[
d_m^{\textrm{weak}} \leq d_m \leq p d_m^{\textrm{weak}}
\]
where $d_m^{\textrm{weak}}$ is as defined in Corollary \ref{weak_corollary}.
\end{thm}

\begin{proof}
The natural action of $\notty$ on the set of characters of type $\langle 1, m \rangle $ yields  $d_m^{\textrm{weak}}$ number of orbits. Consider one such orbit $\notty \cdot \chi$ represented by $\chi$ with  standard expansion as in equation \eqref{exp}; it is precisely the set of characters that are weakly equivalent to $\chi$. On the other hand, by Lemma \ref{second_condition_lemma}, the set $g(0) \notty(x_1,x_2) \chi = \notty(x_1,x_2). \chi$ is precisely the set of characters strictly equivalent to $\chi$. From Proposition \ref{new_addition} the set $g(k) \notty(x_1,x_2) \cdot \chi$ consists of characters strictly equivalent to $\ub{g(k)} \chi$. Finally, assertion \eqref{notty_decomposition} of Proposition \ref{new_addition} implies that the weak equivalence class $\notty \cdot \chi$ is partitioned into \textit{at most} $p$ number of strict equivalence classes; note that the disjointedness of the union in assertion \eqref{notty_decomposition} of Proposition \ref{new_addition} 
is not used.

From Theorem \ref{main_thm}, it follows that the number of conjugacy classes of torsion elements of order $p^2$ and type $\langle 2 ,m \rangle$ is equal to the number of strict equivalence classes of type $\langle 2, m \rangle$; see Definition \ref{secdefn}. The upper bound on the cardinality of the latter set follows from the previous paragraph and Corollary \ref{weak_corollary} .
\end{proof}

\section{Final remarks}

\label{fin_rem}
A complete account of number of strict equivalence classes of arbitrary type $\langle a, m \rangle$ where $a \neq 1,2$ needs to be determined. Moreover, in this article we considered only the case where the base field is a prime field $\F_p$. One may also wonder about the results in the general case of arbitrary finite field.

\section*{Acknowledgments}

I remain grateful to Jonathan Lubin for his patience in responding to numerous emails, and his enthusiasm in clarifying various notions in the theory. Also, I am grateful to the referee for many useful comments and suggestions.

\end{document}